\documentclass[11pt]{amsart}

\usepackage[utf8]{inputenc}  

\usepackage{amsthm}
\usepackage{amsfonts}
\usepackage{amsmath}
\usepackage{amssymb}
\let\amssquare\square
\usepackage{mathtools} 
\usepackage{stmaryrd}
\usepackage{marvosym}
\usepackage{multicol}

\usepackage{xcolor}
\definecolor{darkgreen}{rgb}{0,0.45,0}
\definecolor{darkred}{rgb}{0.75,0,0}
\definecolor{darkblue}{rgb}{0,0,0.6}
\usepackage[colorlinks,citecolor=darkgreen,linkcolor=darkred,urlcolor=darkblue]{hyperref}
\usepackage{breakurl}
\usepackage[mathscr]{eucal}
\usepackage{enumerate} 
\usepackage[capitalize]{cleveref}

\usepackage{mathpartir}

\usepackage{tikz}
\usetikzlibrary{arrows}
\tikzset
{
  diagram/.style=
  {
    matrix of math nodes,
    column sep=2.5em,
    row sep=2.5em,
    text height=1.5ex,
    text depth=.25ex
  },
  cross line/.style={preaction={draw=white,-,line width=6pt}},
  every to/.style={font=\footnotesize},
  cof/.style={>->},                
  fib/.style={->>},                
  inj/.style={right hook->},       
  inj2/.style={left hook->},       
  surj/.style={-open triangle 60}, 
  eq/.style={-,double distance=1.7pt}
}

\makeatletter
\newenvironment{ctikzpicture}
{
  \begingroup
  \smallskip
  \par
  \centering
  \begin{tikzpicture}
}{
  \end{tikzpicture}
  \par
  \smallskip
  \endgroup
  \aftergroup\@afterindentfalse
  \aftergroup\@afterheading
}
\makeatother
\usepackage{tikz-cd}

\theoremstyle{plain}
\newtheorem{theorem}{Theorem}[section]

\newtheorem{proposition}[theorem]{Proposition}
\newtheorem{lemma}[theorem]{Lemma}
\newtheorem{corollary}[theorem]{Corollary}

\theoremstyle{definition}

\newtheorem{example}[theorem]{Example}
\newtheorem{examples}[theorem]{Examples}

\newtheorem{remark}[theorem]{Remark}

\usetikzlibrary{matrix}

\tikzset
{
  diagram/.style=
  {
    matrix of math nodes,
    column sep=2.5em,
    row sep=2.5em,
    text height=1.5ex,
    text depth=.25ex
  },
  every to/.style={font=\footnotesize},
}
\makeatletter
\renewcommand{\paragraph}{\@startsection{paragraph}{4}{0mm}{-0.5\baselineskip}{-1ex}{\bf}}
\makeatother

\newcommand{\Cat}{\mathsf{Cat}} 
\newcommand{\cSet}{\mathsf{cSet}} 
\newcommand{\Set}{\mathsf{Set}} 
\newcommand{\sSet}{\mathsf{sSet}} 
\newcommand{\Top}{\mathsf{Top}} 

\newcommand{\id}{\mathrm{id}}
\newcommand{\op}{\mathrm{op}} 

\newcommand{\bbN}{\mathbb{N}} 

\newcommand{\adjoint}{\dashv} 

\renewcommand{\to}{\rightarrow} 
\newcommand{\into}{\hookrightarrow} 
\newcommand{\iso}{\cong} 

\newcommand{\N}{\mathrm{N}} 

\newcommand{\Cube}{\amssquare} 

\newcommand{\Q}{\mathscr{Q}} 
\newcommand{\T}{\mathrm{T}} 
\newcommand{\U}{\mathrm{U}} 

\newcommand{\chorn}{\sqcap} 

\newcommand{\const}{\mathrm{const}} 
\newcommand{\Sat}{\mathrm{Sat}} 

\begin{document}
\title{A co-reflection of cubical sets into simplicial sets \\ with applications to model structures}
\author{Krzysztof Kapulkin \and Zachery Lindsey \and Liang Ze Wong}
\date{\today}

\begin{abstract}
We show that the category of simplicial sets is a co-reflective subcategory of the category of cubical sets with connections, with the inclusion given by a version of the straightening functor. We show that using the co-reflector, one can transfer any cofibrantly generated model structure in which cofibrations are monomorphisms to cubical sets, thus obtaining cubical analogues of the Quillen and Joyal model structures.
\end{abstract}

\maketitle


Cubical sets are a well-known alternative to simplicial sets in combinatorial homotopy theory.
They were in fact studied by Kan before the introduction of simplicial sets (see, e.g., \cite{kan:abstract-homotopy-i}) and have found manifold applications, including in formal logic \cite{cchm,cisinski:univalent-universes}, directed homotopy theory \cite{krishnan:cubical-approx}, and abstract homotopy theory \cite{CisinskiAsterisque,jardine:categorical-homotopy,maltsiniotis:cube-with-conn-is-test}.

While there is only one version of the simplex category $\Delta$, there are many different versions of the box category $\Cube$, the site for cubical sets. In each case, one takes a certain subcategory of $\Cat$, the category of small categories, generated by the posets of the form $\{0 \leq 1\}^n$. One popular choice, pursued for instance by Cisinski \cite{CisinskiAsterisque} and Jardine \cite{jardine:categorical-homotopy} is to define $\Cube$ as the smallest category containing the \emph{face} and \emph{degeneracy} maps. The drawback of this choice is that the resulting category is not a strict test category (although it is a test category).

In this paper, we consider the category of cubical sets with \emph{connections}, which is known to be a strict test category \cite{maltsiniotis:cube-with-conn-is-test}. A connection is a new kind of degeneracy map that allows us, e.g., to consider a $1$-cube $a \overset{f}{\to} b$ as a degenerate $2$-cube as follows:
 \begin{ctikzpicture}
   \matrix[diagram]
   {
     |(a)| b & |(b)| b \\
     |(c)| a & |(d)| b \\
   };

   \draw[->] (c) to node[left] {$f$} (a);
   \draw[double equal sign distance] (a) to (b);
   \draw[->] (c) to node[above] {$f$} (d);
   \draw[double equal sign distance] (d) to (b);
 \end{ctikzpicture}
This is the minimal category allowing for the definition of the cubical homotopy coherent nerve functor and the Grothendieck construction (also known as (un)straightening).

\subsection*{Contributions.} The first contribution of the present paper is the proof (cf.~\cref{thm:coreflection}) that the straightening-over-the-point functor of \cite{kapulkin-voevodsky:cubical-straightening} defines an inclusion of the category of simplicial sets into the category of cubical sets as a co-reflexive subcategory (with the unstraightening as the co-reflector). The second is a transfer theorem (\cref{thm:model-structure-cSet}) for model structures. Specifically, given a cofibrantly generated model structure on simplicial sets in which each cofibration is a monomorphism, we can right induce (in the sense of \cite{hkrs:transfer-thm,garner-kedziorek-riehl}) a Quillen equivalent model structure on cubical sets. In particular, our theorem gives a model of the homotopy theory of $(\infty,1)$-categories in cubical sets. To our knowledge, this is the first such model.

\subsection*{Organization.} This paper is organized as follows. In \cref{prelims}, we review the background on cubical sets. In \cref{2-functors}, we describe the Grothendieck construction and carefully analyze its left adjoint. \cref{3-coreflection} contains the technical heart of the paper, culminating in the proof that the Grothendieck construction is a co-reflector. Following this, we prove our transfer theorem in \cref{4-induced-model} and discuss the resulting examples in \cref{5-examples}.

\subsection*{Acknowledgements.} We wish to thank Christian Sattler and the anonymous referee for helpful comments.


\section{Cubical sets}\label{prelims}

We write $\Delta$ for the simplex category, i.e., the category whose objects are non-empty finite ordinals $[n] = \{ 0 \leq 1 \leq  \ldots \leq n\}$ and whose maps are monotone functions. The category of simplicial sets, denoted $\sSet$, is the functor category $\Set^{\Delta^\op}$.
We adopt the usual notational conventions regarding simplicial sets, e.g., writing $\Delta^n$ for the representable simplicial sets, $\partial \Delta^n$ for their boundaries, etc.

Similarly, we write $\Cube$ for the \emph{box category with connections}. That is, the objects of $\Cube$ are posets of the form $[1]^n$ and the maps are generated (inside the category of posets) under composition by the following three special classes:
\begin{itemize}
  \item \emph{faces} $\partial^n_{i,\varepsilon} \colon [1]^{n-1} \to [1]^n$ for $i = 1, 2, \ldots , n$ and $\varepsilon = 0, 1$ given by:
  \[ \partial^n_{i,\varepsilon} (x_1, x_2, \ldots, x_{n-1}) = (x_1, x_2, \ldots, x_{i-1}, \varepsilon, x_i, \ldots, x_{n-1})\text{;}  \]
  \item \emph{degeneracies} $\sigma^n_i \colon [1]^n \to [1]^{n-1}$ for $i = 1, 2, \ldots, n$ given by:
  \[ \sigma^n_i ( x_1, x_2, \ldots, x_n) = (x_1, x_2, \ldots, x_{i-1}, x_{i+1}, \ldots, x_n)\text{;}  \]
  \item \emph{connections} $\gamma^n_i \colon [1]^n \to [1]^{n-1}$ for $i = 1, 2, \ldots, n-1$ given by:
  \[ \gamma^n_i (x_1, x_2, \ldots, x_n) = (x_1, x_2, \ldots, x_{i-1}, \max\{ x_i , x_{i+1}\}, x_{i+2}, \ldots, x_n) \text{.} \]
\end{itemize}
To simplify the notation, we will usually omit the superscript $n$ when writing specific face, degeneracy, and connection maps.
We will refer to face maps of the form $\partial_{i,1}$ as \emph{positive face maps} and to those of the form $\partial_{i,0}$ as the \emph{negative face maps}.

Alternatively, one may describe $\Cube$ as the category generated by the above maps subject to the following co-cubical identities (cf.~\cite[(5) and (16)]{grandis-mauri:cubical-sets-and-their-sites}):

\begin{multicols}{2}
\begin{alignat*}{2}
	\partial_{j, \varepsilon} \partial_{i, \varepsilon'} 
		& = \partial_{i+1, \varepsilon'} \partial_{j, \varepsilon} 
		& &	\quad\text{for } j \leq i\text{;}
	\\
	\sigma_i \sigma_j 
		& = \sigma_j \sigma_{i+1} 
		& & \quad \text{for } j \leq i\text{;}		
	\\
	\gamma_j \gamma_i 
		& = \left\{ \begin{array}{l}
			\gamma_i \gamma_{j+1} \\
			\gamma_i \gamma_{i+1}
			\end{array}\right.
		& & \;\;\begin{array}{l}
				\text{for } j > i \text{;} \\
				\text{for } j = i \text{;}
			\end{array}
	\\
	\sigma_j \partial_{i, \varepsilon}
		& = \left\{\begin{array}{l}
				\partial_{i-1, \varepsilon} \sigma_j \\
				\id                                                    \\
				\partial_{i, \varepsilon} \sigma_{j-1}
			\end{array}\right.
		& & \;\;\begin{array}{l}
				\text{for } j < i \text{;} \\
				\text{for } j = i \text{;} \\
				\text{for } j > i \text{;}
			\end{array}
\end{alignat*}
\begin{alignat*}{2}
	&
	\\
	\gamma_j \partial_{i, \varepsilon}
		& = \left\{ \begin{array}{l}
				\partial_{i-1, \varepsilon} \gamma_j \\
				\id                                  \\
				\partial_{i, \varepsilon} \sigma_i   \\
				\partial_{j, \varepsilon} \gamma_{j-1}
			\end{array}\right.
		& & \;\;\begin{array}{l}
				\text{for } j < i-1 \text{;} \\
				\text{for } j = i-1, \, i, \, \varepsilon = 0 \text{;} \\
				\text{for } j = i-1, \, i, \, \varepsilon = 1 \text{;} \\
				\text{for } j > i \text{;} 
			\end{array}
	\\
	\sigma_j \gamma_i
		& = \left\{ \begin{array}{l}
			\gamma_{i-1} \sigma_j   \\
			\sigma_i \sigma_i       \\
			\gamma_i \sigma_{j+1}
		\end{array}\right.
		& & \;\; \begin{array}{l}
			\text{for } j < i \text{;} \\
			\text{for } j = i \text{;} \\
			\text{for } j > i \text{.} 
		\end{array}
\end{alignat*}
\end{multicols}

Clearly, the set $\Cube([1]^m, [1]^n)$ is a subset of all monotone maps $[1]^m \to [1]^n$. The following proposition gives a useful characterization of those monotone functions that are valid morphisms in $\Cube$.

\begin{proposition}[Maltsiniotis, {\cite[Prop.~2.3]{maltsiniotis:cube-with-conn-is-test}}] \label{prop:maltsiniotis}
  A monotone map $f = (f_1, f_2, \ldots, f_n) \colon [1]^m \to [1]^n$ is a morphism in $\Cube$ if and only if each $f_j \colon [1]^m \to [1]$ is of the form:
  \begin{enumerate}
    \item $f_j = \const_0$ (constant function with value $0$);
    \item $f_j = \const_1$;
    \item there exists a subset $A \subseteq \{ 1, 2, \ldots, m\}$ such that $f_j = \max_A$\footnote{i.e., $f_j(x_1, x_2, \ldots, x_n) = \max\{ x_i \ | \ i \in A \}$. Not to be confused with $\max A$, which is the largest $i$ in $A$.} and if for $j < j'$ we have $f_j = \max_A$ and $f_{j'} = \max_{A'}$, then $\max A < \min A'$. \qed
  \end{enumerate}
\end{proposition}

Moreover, using cubical identities, one can derive the following normal forms for all cubical maps.

\begin{theorem}[Grandis-Mauri] \label{normal-form}
  Every map in the category $\Cube$ can be factored uniquely as a composite
  \[ (\partial_{k_1, \varepsilon_1} \ldots \partial_{k_t, \varepsilon_t})
     (\gamma_{j_1} \ldots \gamma_{j_s})
     (\sigma_{i_1} \ldots \sigma_{i_r})\text{,} \]
  where $i_1 > \ldots > i_r \geq 1$, $1 \leq j_1 < \ldots < j_s$, and $k_1 > \ldots > k_t \geq 1$.   
\end{theorem}

\begin{proof}
  This is essentially \cite[Thm.~5.1]{grandis-mauri:cubical-sets-and-their-sites} with the opposite ordering of degeneracy maps, which does not affect the statement.
\end{proof}


We write $\cSet$ for the resulting category of cubical sets, i.e., contravariant functors $\Cube^\op \to \Set$. Following the usual conventions for simplicial sets, we write $\Cube^n$ for the representable cubical sets, represented by $[1]^n$.

The cartesian product of cubical sets is homotopically well-behaved; however, one does not have $\Cube^m \times \Cube^n \iso \Cube^{m+n}$. Thus instead we consider the \emph{geometric product} defined via the left Kan extension of the functor $\Cube \times \Cube \to \cSet$ taking $([1]^m , [1]^n)$ to $\Cube^{m+n}$ along the Yoneda embedding as in
 \begin{ctikzpicture}
   \matrix[diagram]
   {
     |(a)| \Cube \times \Cube & |(c)| \cSet \\
     |(b)| \cSet \times \cSet  &  \\
   };

   \draw[inj] (a) to (b);
   \draw[->] (a) to (c);
   \draw[->] (b) to node[below]  {$\otimes$} (c);
 \end{ctikzpicture}
The geometric product defines a monoidal structure on $\cSet$ and we will work with this, rather than the cartesian structure, throughout the paper.


\section{Co-reflection: construction} \label{2-functors}

The goal of this section is to define the functors forming the proposed co-reflection, i.e., an adjunction $\sSet \rightleftarrows \cSet$ with fully faithful left adjoint. This is a special case of the Grothendieck construction of \cite[\S3]{kapulkin-voevodsky:cubical-straightening}. Specifically, the co-reflector will be given by the Grothendieck construction over the point, i.e., $\int_{\Delta^0}$ in the notation of \cite{kapulkin-voevodsky:cubical-straightening}.

However, the variant of the box category $\Cube$ used in \cite{kapulkin-voevodsky:cubical-straightening} differs from ours, as it is taken to be the full subcategory of posets on objects of the form $[1]^n$. Although the necessary results of \cite[\S2-3]{kapulkin-voevodsky:cubical-straightening} are true for more restrictive choices of the box category such as the one considered here, we prefer not to rely on such results and will instead describe the co-reflection directly.

We will construct an adjoint pair of the form
\[ \Q \colon \sSet \rightleftarrows \cSet : \! \int\text{,}\]
where $\Q$ arises as the left Kan extension of a cosimplicial object $Q^\bullet \colon \Delta \to \cSet$ which we now describe.
For $n \in \bbN$ and $0 < i < n$, there is a canonical map $\partial^n_{i, 1} \colon \Cube^{i-1} \otimes \Cube^{n-i} = \Cube^{n-1} \to \Cube^n$ (i.e., the positive $i^{\text{th}}$-face). This induces a map $\displaystyle{\bigcup\limits_{0 < i < n}}  \Cube^{i-1} \otimes \Cube^{n-i} \to \Cube^n$ and we define $Q^n$ as the pushout:

 \begin{ctikzpicture}
   \matrix[diagram]
   {
     |(a)| \displaystyle{\bigcup\limits_{0 < i < n}}  \Cube^{i-1} \otimes \Cube^{n-i} & |(b)| \Cube^n \\
     |(c)| \displaystyle{\bigcup\limits_{0 < i < n}}  \Cube^{i-1}  & |(d)| Q^n \\
   };

   \draw[->] (a) to (b);
   \draw[->] (a) to (c);
   \draw[->] (b) to (d);
   \draw[->] (c) to (d);
 \end{ctikzpicture}

where the vertical map is induced by projecting off the last $n-i$ entries. Thus $Q^n$ is a quotient of $\Cube^n$. More precisely, we may define an equivalence relation $\sim$ on the set $\Cube^n_m$ of $m$-cubes of the combinatorial $n$-cube as the reflexive closure of:
\[ (f_1, \ldots , f_n) \sim (g_1, \ldots, g_n) \quad \text{iff} \quad \text{there is } j \leq n \text{ such that } f_1 = g_1, \ \ldots,\  f_{j-1} = g_{j-1},\ f_j = g_j = \const_1. \]

\begin{proposition} \label{prop:Qn} \leavevmode
  \begin{enumerate}
    \item\label{prop:Qn-quotient}  The set of $m$-cubes of $Q^n$ is the quotient $\Cube^n_m/\!\sim$.
    \item\label{prop:Qn-unique-representation} In particular, every $m$-cube of $Q^n$ has a unique representation as a sequence
    \[(f_1, f_2, \ldots, f_j, \const_1, \ldots, \const_1)\text{,}\]
    where $f_1, f_2, \ldots, f_j \neq \const_1$.
  \end{enumerate}
\end{proposition} 

\begin{proof}
 \cref{prop:Qn-quotient} is clear by the definition of $Q^n$. \cref{prop:Qn-unique-representation} follows from \cref{prop:Qn-quotient} and \cref{prop:maltsiniotis}.
\end{proof}

We will write $\pi_n \colon \Cube^n \to Q^n$ for the quotient map.

\begin{examples}
For $n=0, 1, 2$, we can describe/depict $Q^n$'s as follows:
\begin{itemize}
  \item $Q^0 = \Cube^0$;
  \item $Q^1 = \Cube^1$;
  \item $Q^2 =\vcenter{
 \begin{tikzpicture}
   \matrix[diagram]
   {
     |(a)| \bullet & |(b)| \bullet \\
     |(c)| \bullet & |(d)| \bullet \\
   };

   \draw[->] (a) to (b);
   \draw[->] (c) to (a);
   \draw[double equal sign distance] (d) to (b);
   \draw[->] (c) to (d);
 \end{tikzpicture} }$.
\end{itemize}
Similarly, $Q^3$ can be obtained as a quotient of $\Cube^3$, contracting one of the squares to a point and one of the remaining squares to a line.
\end{examples} 

\begin{proposition}
  The assignment $[n] \mapsto Q^n$ extends to a cosimplicial object $Q^\bullet \colon \Delta \to \cSet$.
\end{proposition}

\begin{proof}
The remaining face maps $\Cube^{n-1} \to \Cube^n$ (that is, $\partial^n_{n,1}$ and $\partial^n_{i, 0}$ for $i = 1, \ldots, n$), the last degeneracy $\sigma_n \colon \Cube^n \to \Cube^{n-1}$, and the connections $\gamma_j \colon \Cube^n \to \Cube^{n-1}$ descend to maps between the corresponding $Q^n$'s, yielding a co-simplicial object $Q^\bullet \colon \Delta \to \cSet$. This correspondence is as follows:

$\begin{array}{l|ccccccc}
\text{a map } Q^{n-1} \to Q^n & 
0^\text{th} \text{ face} & 1^\text{st} \text{ face} & 2^\text{nd} \text{ face} & 
\cdots & j^\text{th} \text{ face} & \cdots & n^\text{th} \text{ face} \\ \hline

\text{is induced by a map } \Cube^{n-1} \to \Cube^n & 
\partial_{n, 1} & \partial_{n, 0} & \partial_{n-1,0} &
\cdots & \partial_{n-j+1, 0} & \cdots & \partial_{1, 0} \\
& &&&&&& \\
\text{a map } Q^n \to Q^{n-1} & 
0^\text{th} \text{ deg.} & 1^\text{st} \text{ deg.} & 2^\text{nd} \text{ deg.} &
\cdots & j^\text{th} \text{ deg.} & \cdots & (n-1)^\text{st} \text{ deg.} \\ \hline

\text{is induced by a map } \Cube^n \to \Cube^{n-1} & 
\sigma_n & \gamma_{n-1}  & \gamma_{n-2} &
\cdots & \gamma_{n-j} & \cdots & \gamma_1
\end{array}$

The verification that these indeed obey the co-simplicial identities (i.e., form a co-simplicial object) is straightforward using the co-cubical identities and the equivalence relations defining the $Q^n$'s.
For instance, the co-simplicial identity $\partial_{1} \partial_0 = \partial_0 \partial_0$ follows from
\[
	\partial_1 \partial_0 := \partial_{n+1,0} \partial_{n,1} \sim \partial_{n+1,1} \partial_{n,1}  =: \partial_0 \partial_0\text{,}
\]
whereas the co-simplicial identities away from index $0$ do not require the equivalence relation defining $Q^n$.
\end{proof}

\begin{remark}
The other degeneracy maps, (i.e., $\sigma_i$ for $i = 1, \ldots, n-1$) do not descend to maps between $Q^n$'s, since they do not respect the equivalence relation $\sim$ used in the definition of $Q^n$.
\end{remark}

\begin{lemma} \label{lem:Qff}
	$Q ^\bullet\colon \Delta \to \cSet$ is full and faithful.
\end{lemma}

\begin{proof}
  Using the above characterization of maps between $Q^n$'s, one easily checks that the cubical maps that descend to maps $Q^m \to Q^n$ are exactly those that can be written as composites of maps arising from $\Delta$.
\end{proof}

For $X \in \cSet$, define $\int X \iso \cSet(Q^\bullet, X)$. This gives a functor $\int \colon \cSet \to \sSet$ whose left adjoint, denoted $\Q$, is given by the left Kan extension of $Q^\bullet$ along the Yoneda embedding $\Delta \into \sSet$.

\begin{remark}
  Although it is non-obvious, the functor $\Q \colon \sSet \to \cSet$ does not preserve products. In general, the map $\Q(A \times B) \to \Q A \times \Q B$ is a monomorphism. However already in the case of $A = B = \Delta^1$, it is not an isomorphism.
\end{remark}


\section{Co-reflection: proof} \label{3-coreflection}

In this section, we show that the unit $\eta$ of the adjunction $\Q \dashv \int$ is a natural isomorphism, establishing $\sSet$ as a co-reflective subcategory of $\cSet$ (cf. \cref{thm:coreflection}).
We begin with a very general criterion for pushouts.

\begin{lemma}\label{lem:pushout-episquare}
	In any category, suppose we have the following commuting diagram
	\[
		\begin{tikzcd}[sep = large]
			B \ar[r, "s_1"] \ar[d, "p_3"', two heads] \ar[rr, bend left, equals] & A \ar[r, "p_1", two heads] \ar[d, "p_2"', two heads] & B \ar[d, "p_3", two heads]
			\\
			D \ar[r, "s_4"] & C \ar[r, "p_4", two heads] & D
		\end{tikzcd}
	\]
	where all $p_i$'s are epimorphisms. 
	Then the right-hand square is a pushout square.
\end{lemma}

\begin{proof}
	Note that $s_1$ being a section of $p_1$ implies that $s_4$ is a section of $p_4$ as well.
	Consider the commutative diagram of solid arrows:
	\[
		\begin{tikzcd}[sep = large]
			B \ar[r, "s_1"] \ar[d, "p_3"', two heads] \ar[rr, bend left, equals] & A \ar[r, "p_1", two heads] \ar[d, "p_2"', two heads] & B \ar[d, "p_3", two heads] \ar[ddr, bend left, "x"]
			\\
			D \ar[r, "s_4"] & C \ar[r, "p_4", two heads] \ar[drr, bend right, "y"'] & D \ar[dr, dashed, "y\,s_4"']
			\\
			& & & X
		\end{tikzcd}
	\]
	Then $y\,s_4\,p_3 = x$, so $y\,p_2 = x\,p_1 = y\,s_4\,p_3\,p_1 = y\,s_4\,p_4\,p_2$.
	Since $p_2$ is an epimorphism, we obtain $y = y\,s_4\,p_4$, so the diagram with the dashed arrow also commutes.
	Since the map $p_3\,p_1 = p_4\,p_2$ is an epimorphism, the solution $y\, s_4$ is unique.
\end{proof}


The next three lemmas deal with the combinatorics of cubical sets.

Fix subsets $A, B \subseteq \{1,2, \ldots, k\}$. 
Let $m = k - |A|, n = k - |B|$, and $\ell = k - |A \cup B|$.
Write $\sigma_A$ for the composite of degeneracies $\sigma_{i_1} \dots \sigma_{i_m} \colon \Cube^k \to \Cube^m$ for $i_j \in A$, and $\partial_A$ for the positive face map $\Cube^m \to \Cube^k$ that is a section of $\sigma_A$, and similarly for other subsets.
All indices will be with respect to the ambient set $\{1, 2, \dots, k\}$, so a $p$-cube in $\Cube^m$  will be denoted $(f_{i_1}, f_{i_2}, \ldots, f_{i_m})$ where $i_1, \ldots, i_m \notin A$.
 
\begin{lemma} \label{lem:pushout-cube}
	The following diagram is a pushout:
	\[
		\begin{tikzcd}[sep = large]
			\Cube^k \ar[r, "\sigma_A", two heads] \ar[d, "\sigma_B"', two heads] & \Cube^m \ar[d, "\sigma_{B\setminus A}", two heads]
			\\
			\Cube^n \ar[r, "\sigma_{A \setminus B}", two heads] & \Cube^\ell
		\end{tikzcd}
	\] 
\end{lemma}

\begin{proof}
	Take the sections to be the positive face maps $\partial_A$ and $\partial_{A \setminus B}$.
	The cubical identities ensure that the conditions of \cref{lem:pushout-episquare} are satisfied.
\end{proof}

Keeping $A$ and $B$ as before, recall the symmetric difference $A \triangle B := (A \setminus B) \cup (B \setminus A)$.
Let 
\[ 
	C =	\left\{
				\begin{array}{ll}
						\big\{\min A \triangle B, \ldots, k \big\} \cup A \cup B
 & \text{if } A \neq B \text{;} \\
					A & \text{otherwise}  \text{;}
				\end{array}
			\right.	
\] 
and let $r = k -|C|$.
By construction, the degeneracy $\sigma_{C\setminus A} \colon \Cube^m \to \Cube^r$ descends to an epimorphism $\bar{\sigma}_{C \setminus A} \colon Q^m \to Q^r$, and the positive face map $\partial_{C \setminus A}$ descends to a section  $\bar{\partial}_{C \setminus A}$ of  $\sigma_{C\setminus A}$.
Similarly, we have an epimorphism $\bar{\sigma}_{C \setminus B} \colon Q^n \to Q^r$ with a section $\bar{\partial}_{C \setminus B}$. 

\begin{lemma} \label{lem:pushout-Q}
	The following diagram is a pushout:
	\[
		\begin{tikzcd}[sep = large]
			\Cube^k \ar[r, "\pi_m \sigma_A", two heads] \ar[d, "\pi_n \sigma_B"', two heads] & Q^m \ar[d, "\bar{\sigma}_{C \setminus A}", two heads]
			\\
			Q^n \ar[r, "\bar{\sigma}_{C \setminus B}", two heads] & Q^r
		\end{tikzcd}
	\]	
\end{lemma}

\begin{proof}
	If $A = B$, then $\pi_m \sigma_A = \pi_n \sigma_B$, and $\bar{\sigma}_{C \setminus A} = \bar{\sigma}_{C \setminus B}$ is the identity on $Q^m = Q^n = Q^r$, so the diagram is a pushout.

	If $A \neq B$, we may assume without loss of generality that $\min A \triangle B \in B \setminus A$.
	Since pushouts in $\cSet$ are computed pointwise, it suffices to show that following diagram is a pushout for all $p$, where we use the same notation for the induced maps of $p$-cubes:
	\[
		\begin{tikzcd}[sep = large]
			\Cube^k_p \ar[r, "\pi_m \sigma_A", two heads] \ar[d, "\pi_n \sigma_B"', two heads] & Q^m_p \ar[d, "\bar{\sigma}_{C \setminus A}", two heads]
			\\
			Q^n_p \ar[r, "\bar{\sigma}_{C \setminus B}", two heads] & Q^r_p
		\end{tikzcd}
	\]
	By \cref{prop:Qn}, each element in $Q^n_p$ is of the form 
	\[
		f = (f_{i_1}, f_{i_2}, \ldots, f_{i_j}, \const_1, \ldots, \const_1)
	\]
	where $f_{i_\ell} \neq \const_1$ if $\ell \leq j$.
	Let $\rho_n \colon Q^n_p \to \Cube^n_p$ denote the function sending $f \in Q^n_p$ to itself in $\Cube^n_p$.
	This is a section of $\pi_n \colon \Cube^n_p \to Q^n_p$, so the composite 
	\[
		\begin{tikzcd}
			\hat{\partial}_B : Q^n_p \ar[r, "\rho_n"] & \Cube^n_p \ar[r, "\partial_{B}"] & \Cube^k_p
		\end{tikzcd}
	\]
	is a section of $\pi_n \sigma_B \colon \Cube^k_p \to Q^n_p$.
	Note that $\rho_n$ and $\hat{\partial}_B$ do not arise from maps of cubical sets.

	By \cref{lem:pushout-episquare}, it suffices to verify that the following diagram commutes:
	\[
		\begin{tikzcd}[sep = large]
			Q^n_p \ar[d, "\bar{\sigma}_{C \setminus B}"', two heads] \ar[r, "\hat{\partial}_B"] & \Cube^k_p \ar[d, "\pi_m \sigma_A", two heads] 
			\\
			Q^r_p \ar[r, "\bar{\partial}_{C \setminus A}"]	& Q^m_p
		\end{tikzcd}
	\]
	Let $f$ be a $p$-cube in $Q^n_p$, and let $g =   \pi_m\, \sigma_{A} \, \hat{\partial}_B \, f$ and $h = \bar{\partial}_{C \setminus A}\, \bar{\sigma}_{C \setminus B}\, f$ in $Q^m_p$. 
	Then
	\[
		g_i = \left\{
				\begin{array}{ll}
					 f_i & \text{if } i \notin C \text{;} \\
					 \const_1 & \text{otherwise}  \text{;}
				\end{array}
			\right.
		\quad \quad \quad \quad
		h_i = \left\{
				\begin{array}{ll}
					f_i & \text{if } i \notin A \cup B \text{;} \\
					\const_1 & \text{otherwise}  \text{.}
				\end{array}
			\right.		
	\]	
	For $i \notin A$ such that $i  < \min A \triangle B$, we have $i \notin C \supseteq A \cup B$, so $g_i = h_i = f_i$.
	For $i = \min A \triangle B$, which is in $B \setminus A$ by assumption, we have $i \in A \cup B \subseteq C$, so $g_i = h_i = \const_1$.
	But this identifies $g$ with $h$ in $Q^m_p$, thus the diagram commutes.
\end{proof}

\begin{lemma} \label{cor:pushout-Q-any}
	Any square of the form
	\[
		\begin{tikzcd}
			\Cube^k \ar[r, ] \ar[d, ] & Q^m \ar[d,]
			\\
			Q^n \ar[r, ] & X
		\end{tikzcd}
	\]
	can be factored as
	\[
		\begin{tikzcd}
			\Cube^k \ar[r, , two heads] \ar[d,, two heads] \ar[dr, phantom, very near end, "\ulcorner"] & Q^{m'} \ar[d, , two heads] \ar[r, ] & Q^m \ar[dd, ]
			\\
			Q^{n'} \ar[r, , two heads] \ar[d, ] & Q^r \ar[dr] & 
			\\
			Q^n \ar[rr,] & & X
		\end{tikzcd}
	\]
	where the pushout square consists of maps induced by degeneracies.
\end{lemma}

\begin{proof}
	By \cref{normal-form}, any map $\Cube^k \to \Cube^m$ may be factored as a degeneracy $\Cube^k \to \Cube^{m'}$ followed by a map $\Cube^{m'} \to \Cube^m$ which descends to a map $Q^{m'} \to Q^m$.
	Factor $\Cube^k \to \Cube^n$ in a similar fashion, then apply \cref{lem:pushout-Q}.
\end{proof}

Using the above lemma, we can now show that the functor $\Q \colon \sSet \to \cSet$ is faithful. The technical part is contained in the following statement.

\begin{proposition} \label{prop:Q-faithful-on-simp}
  Given $x, y \colon \Delta^n \to X$, if $\Q x = \Q y$, then $x = y$, i.e., $\Q$ induces an injective map $\sSet(\Delta^n, X) \to \cSet(Q^n, \Q X)$.
\end{proposition}

The proof requires the following lemma.

\begin{lemma} \label{lem:no-lift-Qn}
  There is no map $\Cube^n \to \Q(\partial \Delta^n)$ making the following diagram commute
  \begin{ctikzpicture}
   \matrix[diagram]
   {
     |(a)|  	                               & |(b)| \Cube^n \\
     |(c)| 	\Q(\partial \Delta^n)  & |(d)| Q^n \\
   };

   \draw[->>] (b) to node[right] {$\pi_n$} (d);
   \draw[->]   (c) to (d);
   \draw[->]   (b) to (c);
 \end{ctikzpicture}
\end{lemma}

\begin{proof}
  Immediate, since any map $\Cube^n \to \Q(\partial \Delta^n)$ would need to factor through an $(n-1)$-dimensional face.
\end{proof}

\begin{proof}[Proof of \cref{prop:Q-faithful-on-simp}]
  This is proven by skeletal induction with respect to $X$. The conclusion is clear for $n=0$, i.e., when both $x$ and $y$ are points of $X$.
  
  If both $x$ and $y$ are degenerate, then the conclusion follows directly by the inductive hypothesis. Otherwise, if say $x$ is non-degenerate, then the fact that $\Q x = \Q y$ while $x \neq y$ contradicts \cref{lem:no-lift-Qn}.
\end{proof}

\begin{corollary} \label{cor:Q-faithful}
  The functor $\Q \colon \sSet \to \cSet$ is faithful. \qed
\end{corollary}

\begin{lemma} \label{lem:eta-iso}
	For each $X \in \sSet$, the unit $\eta_X \colon X \to \int \Q X$ is an isomorphism.
\end{lemma}

\begin{proof}
  By \cref{cor:Q-faithful}, it suffices to give a section of the map $\sSet(\Delta^n, X) \to \cSet(Q^n, \Q X)$.
	

	Given $\varphi \colon Q^k \to \Q X$, we first precompose with $\pi_k \colon \Cube^k \to Q^k$ to obtain $\varphi\, \pi_k \colon \Cube^k \to \Q X$.
	We factor $\varphi \, \pi_k$ through one of the components of the colimit defining $\Q X$ to obtain the following square on the left, then apply \cref{cor:pushout-Q-any} to obtain the square on the right:
	\[
		\begin{tikzcd}
			\Cube^k \ar[d, "\pi_k"', two heads] \ar[r, "f"] & Q^n \ar[d, "\Q x"]
			\\
			Q^k \ar[r, "\varphi"] & \Q X
		\end{tikzcd}
		\quad \quad = \quad \quad
		\begin{tikzcd}
			\Cube^k \ar[r, two heads] \ar[d, two heads] \ar[dr, very near end, phantom, "\ulcorner"] & Q^{n'} \ar[r] \ar[d, two heads]  & Q^n \ar[d, "\Q x"]
			\\
			Q^k \ar[r, two heads]  & Q^{r} \ar[r] & \Q X
		\end{tikzcd}		
	\]
	Taking the positive face map $\partial \colon Q^r \to Q^{n'}$ yields a factorization of $\varphi$ as
	\[
		\begin{tikzcd}
			Q^k \ar[r, two heads] & Q^{r} \ar[r, "\partial"] & Q^{n'} \ar[r] & Q^n \ar[r, "\Q x"] & \Q X
		\end{tikzcd}
	\]
	By \cref{lem:Qff}, the map $Q^k \to Q^n$ is of the form $\Q f$ for some $f \colon \Delta^k \to \Delta^n$.
	We may then factor $xf \colon \Delta^k \to X$ uniquely as a degenerate $g \colon \Delta^k \to \Delta^m$ followed by a non-degenerate $y \colon \Delta^m \to X$, so that $\varphi = \Q y \circ \Q g$.
	Note that this is independent of the choice of $\partial$ or $f$, so that we have a well-defined function $\varphi \mapsto yg$,
which is the desired section.
\end{proof}

This gives the main theorem of this section.

\begin{theorem} \label{thm:coreflection}
  The functors $\Q \adjoint \int$ define a co-reflective inclusion of $\sSet$ into $\cSet$. \qed
\end{theorem}


\section{Induced model structures} \label{4-induced-model}

Given any model structure on $\sSet$, we declare a map $f$ in $\cSet$ to be:
\begin{itemize}
  \item a \emph{fibration} if $\int f$ is a fibration of simplicial sets;
  \item a \emph{weak equivalence} if $\int f$ is a weak equivalence of simplicial sets;
  \item a \emph{cofibration} if it has the left lifting property with respect to acyclic fibrations, as defined above.
\end{itemize}

If the above three classes of maps define a model structure on $\cSet$, we refer to such a model structure as \emph{right induced} by $\int$. The goal of this section is to prove the following theorem:

\begin{theorem}\label{thm:model-structure-cSet}
  Given any cofibranty generated model structure on $\sSet$ in which every cofibration is a monomorphism, the adjunction $\Q{:}\,\sSet \rightleftarrows \cSet\,{:}\! \int$ right induces a Quillen equivalent model structure on $\cSet$.
\end{theorem}

We precede the proof with several categorical lemmas.

\begin{lemma} \label{lem:epsilon-mono}
	For any $X \in \cSet$, the counit $\varepsilon_X \colon \Q \int X \to X$ is a monomorphism. 
\end{lemma}

\begin{proof}
Unwinding the definitions, we see that $k$-cubes of $\Q\int X$ are represented by composable pairs of the form $\Cube^k \to Q^n \to X$.
Two such $k$-cubes are identified by $\varepsilon_X$ if they fit into a commutative square of the form
\begin{ctikzpicture}
   \matrix[diagram]
   {
     |(a)| \Cube^k	& |(b)| Q^n \\
     |(c)| Q^m		    & |(d)| X \\
   };

   \draw[->] (a) to (b);
   \draw[->] (a) to (c);
   \draw[->] (b) to (d);
   \draw[->] (c) to (d);
 \end{ctikzpicture}
This square can be factored as in \cref{cor:pushout-Q-any}, which shows that the two $k$-cubes of $\Q\int X$ are identified in the colimit.
\end{proof}

\begin{lemma} \label{lem:int-po}
  The functor $\int \colon \cSet \to \sSet$ preserves pushouts of two monomorphisms.
\end{lemma}

\begin{proof}
  Consider a pushout square in $\cSet$ where $A \to B_i$ are monomorphisms:
  \begin{ctikzpicture}
   \matrix[diagram]
   {
     |(a)| A	  & |(b)| B_1 \\
     |(c)| B_2 & |(d)| P \\
   };

   \draw[->] (a) to (b);
   \draw[->] (a) to (c);
   \draw[->] (b) to (d);
   \draw[->] (c) to (d);
 \end{ctikzpicture}
  The pushout inclusions are monomorphisms and $Q^n$ is a quotient of a representable. Hence any map $Q^n \to P$ must factor through one of the inclusions $B_i \into P$. It follows that each of the functors $\cSet(Q^n , -) \colon \cSet \to \Set$ preserves this pushout. Since colimits in $\sSet$ are computed pointwise, $\int$ preserves this pushout as well.
\end{proof}

\begin{lemma} \label{lem:int-comp}
  The functor $\int \colon \cSet \to \sSet$ preserves transfinite compositions.
\end{lemma}

\begin{proof}
  It suffices to show that the result holds pointwise, i.e., for functors $\cSet (Q^n, -) \colon \cSet \to \Set$. Each $Q^n$ is compact, as a quotient of a representable, and hence $\cSet (Q^n, -)$ preserves filtered colimits.
\end{proof}

\begin{lemma} \label{lem:Q-mono}
  The functor $\Q \colon \sSet \to \cSet$ preserves monomorphisms.
\end{lemma}

\begin{proof}
  Immediate by induction on skeleta.
\end{proof}

At this point, we fix a model structure on $\sSet$ and let $J_\Delta$ be the generating set of its acyclic cofibrations. We set $J = \Q(J_\Delta)$ and run the Small Object Argument on $J$ to generate a factorization system $(\Sat(J), \mbox{RLP}(J))$ on $\cSet$. 

\begin{lemma} \label{lem:ac-pushouts}
  Let $A \to B$ be an acyclic cofibration of simplicial sets and let
   \begin{ctikzpicture}
   \matrix[diagram]
   {
     |(a)| \Q A	& |(b)| X \\
     |(c)| \Q B  & |(d)| Y \\
   };

   \draw[->] (a) to (b);
   \draw[->] (a) to (c);
   \draw[->] (b) to (d);
   \draw[->] (c) to (d);
 \end{ctikzpicture}
 be a pushout square in $\cSet$. Then the map $X \to Y$ is a weak equivalence (i.e., its image under $\int$ is a weak equivalence).
\end{lemma}

\begin{proof}
  Applying $\int$ to the span $ \Q B \leftarrow \Q A \to X$ and taking the pushout, we obtain a diagram
     \begin{ctikzpicture}
   \matrix[diagram]
   {
     |(a)| A	& |(b)| \int X \\
     |(c)| B  & |(d)| B \cup_A \int X \\
   };

   \draw[->] (a) to (b);
   \draw[->] (a) to (c);
   \draw[->] (b) to (d);
   \draw[->] (c) to (d);
 \end{ctikzpicture}
 and, in particular, $\int X \to B \cup_A \int A$ is an acyclic cofibration. We use its image under $\Q$ to factor the original square
      \begin{ctikzpicture}
   \matrix[diagram]
   {
     |(a)| \Q A	& |(b)| \Q \int X                    & |(c)| X \\
     |(d)| \Q B  & |(e)| \Q (B \cup_A \int X) & |(f)| Y  \\
   };

   \draw[->] (a) to (b);
   \draw[->] (a) to (d);
   \draw[->] (b) to (c);
   \draw[->] (b) to (e);
   \draw[->] (c) to (f);
   \draw[->] (d) to (e);
   \draw[->] (e) to (f);
 \end{ctikzpicture}
 Since $\Q$ is a left adjoint, the left hand square is a pushout and hence by the pasting lemma for pushouts (the formal dual of the pasting lemma for pullbacks, cf.~\cite[Ex.~III.4.8]{mac-lane:cwm})  so is the right hand square. Moreover, the right hand square is a pushout of monomorphisms (by \cref{lem:epsilon-mono,lem:Q-mono}) and hence it is preserved by $\int$.
  
  Thus by \cref{lem:eta-iso}, the map $\int X \to \int Y$ is isomorphic to $\int X \to B \cup_A \int X$, hence an equivalence.
\end{proof}

\begin{lemma}\label{lem:Sat-is-weq}
  Every morphism in $\Sat(J)$ is a weak equivalence of cubical sets.
\end{lemma}

\begin{proof} 
  The class $\Sat(J)$ is obtained by closing the set $J$ under retracts, pushouts, and transfinite compositions.
  Each morphism in $J$ is a weak equivalence by \cref{lem:eta-iso}.
  The closure under retracts is clear, the closure under pushouts follows from \cref{lem:ac-pushouts}, and the closure under transfinite composition by \cref{lem:int-comp} and the analogous property for simplicial sets.
\end{proof}

\begin{proof}[Proof of \cref{thm:model-structure-cSet}]
By \cite[Cor.~3.1.7]{hkrs:transfer-thm}, any cofibrantly generated model structure on $\sSet$ is an accessible model structure.
Using \cite[Prop.~2.1.4.(1)]{hkrs:transfer-thm},\footnote{Although the article \cite{hkrs:transfer-thm} contains an error, it was fixed in \cite{garner-kedziorek-riehl}, and thus its results can be applied in our setting.} to obtain the right induced model structure, it suffices to verify that maps with the left lifting property with respect to fibrations are weak equivalences, which is exactly the statement of \cref{lem:Sat-is-weq}.

%

The functor $\int$ is a right Quillen functor by the definition of the model structure on $\cSet$.
The unit of $\Q \adjoint \int$ is a weak equivalence by \cref{lem:eta-iso}.
Applying \cref{lem:eta-iso}, we also see that for any cubical set $X$, the map $\int \varepsilon_X \colon \int \Q \int X \to \int X$ is an isomorphism, and hence the counit is a weak equivalence as well.
\end{proof}


\section{Examples} \label{5-examples}

By \cref{thm:model-structure-cSet}, we immediately obtain the following:

\begin{corollary}
  Both the Joyal and the Quillen model structures on $\sSet$ right induce Quillen equivalent model structures on $\cSet$. \qed
\end{corollary}

Let $\cSet_{IJ}$ and $\cSet_{IQ}$ denote these model structures, respectively. The following diagram summarizes the four model structures involved:
\begin{ctikzpicture}
   \matrix[diagram]
   {
     |(a)| \sSet_Q	& & |(b)| \cSet_{IQ} \\
     && \\
     |(c)| \sSet_J & & |(d)| \cSet_{IJ} \\
   };

   \draw[->] (a) to [bend left=8] node[above]  {$\Q$} (b);
   \draw[->] (a) to [bend left=12] node[right]  {$\id$} (c);
   \draw[->] (b) to [bend left=12] node[right]  {$\id$} (d);
   \draw[->] (c) to [bend left=8] node[above]  {$\Q$} (d);
   
   \draw[->] (b) to [bend left=8] node[below] {$\int$} (a);
   \draw[->] (c) to [bend left=12] node[left] {$\id$} (a);
   \draw[->] (d) to [bend left=12] node[left] {$\id$} (b);
   \draw[->] (d) to [bend left=8] node[below] {$\int$} (c);
 \end{ctikzpicture}
where all adjunctions are Quillen adjunctions and the horizontal functors are Quillen equivalences.

Cofibrations in these model structures are always monomorphisms, since by adjointness they are generated by the images of boundary inclusions $\partial \Delta^n \into \Delta^n$ under $\Q$. 
However not all monomorphisms are cofibrations and in fact very few cubical sets are cofibrant.

\begin{example}
  The cubical set $\Cube^2$ is cofibrant in neither $\cSet_{IJ}$ nor $\cSet_{IQ}$. Indeed, by construction each $2$-cube of a cofibrant cubical set has a degenerate face among its four main faces, which is not the case for $\Cube^2$.
\end{example}



To our knowledge, the model structure $\cSet_{IJ}$ is the first model structure on $\cSet$ presenting the homotopy theory of $(\infty,1)$-categories. However, there is a well-established model structure on $\cSet$ for the homotopy theory of $\infty$-groupoids, namely, the Grothendieck model structure, denoted $\cSet_G$. 

In the remainder of this section, we will show that the adjoint pair of identity functors defines a Quillen equivalence between $\cSet_{G}$ and $\cSet_{IQ}$.
We begin by describing the Grothendieck model structure.
Following \cite[Thm.~1.7]{cisinski:univalent-universes}, it is a cofibrantly generated model structure in which the cofibrations are the monomorphisms and fibrations have the right lifting property with respect to the open box inclusions $\chorn^n_{i,\varepsilon} \to \Cube^n$ (open boxes are defined in the standard way).

For our purposes however, it is better to see the Grothendieck model structure as left induced by a certain functor $\cSet \to \sSet$, which we shall next describe.

The embedding $\Cube \into \Cat \overset{\N}{\to} \sSet$ defines a co-cubical object in the category of simplicial sets, explicitly given by $[1]^n \mapsto (\Delta^1)^n$. This yields an adjoint pair
\[ \T \colon \cSet \rightleftarrows \sSet : \! \U \]
with $\T$ given by the left Kan extension of $\Cube \into \sSet$ along the Yoneda embedding, and $(\U X)_n = \cSet((\Delta^1)^n, X)$.
By \cite[Prop.~8.4.28 and Lem.~8.4.29]{CisinskiAsterisque}, one sees that the Grothendieck model structure on $\cSet$ is indeed left induced by $\T$ and further, by \cite[Thm.~8.4.30]{CisinskiAsterisque}, $\T$ is in fact a Quillen equivalence.


Thus we can compare the two model structures for $\infty$-groupoids on $\cSet$ directly.

\begin{proposition}
  The adjunction $\id \colon \cSet_{IQ} \rightleftarrows \cSet_G :\! \id$ is a Quillen equivalence.
\end{proposition}

\begin{proof}
  As noted above, the cofibrations in the induced model structure are monomorphisms and hence cofibrations in the Grothendieck model structure.
  It thus suffices to check that the maps $\Q\Lambda^n_i \to Q^n$ are weak equivalences in the Grothendieck model structure.
  This follows by showing that both $\T\Q\Lambda^n_i$ and $\T Q^n$ are contractible simplicial sets.
  Indeed, using the fact that both $\T$ and $\Q$, as well as the geometric realization functor $|-| \colon \sSet \to \Top$ are left adjoints, we see that $|\T Q^n|$ is a quotient of $[0,1]^n$ homeomorphic to $\Delta_n$ (the topological simplex), whereas $|\T\Q\Lambda^n_i|$ is homeomorphic to $|\Lambda^n_i|$ (the topological horn).
\end{proof}




%
%
%
%





\bibliographystyle{amsalphaurlmod}
\bibliography{general-bibliography}

\end{document}